\documentclass[12pt,a4paper]{article}
\usepackage{amsmath,amssymb,amstext,amsfonts,latexsym,amsthm}
\usepackage{dsfont,mathrsfs,txfonts}
\usepackage{graphicx,float,hyperref,color}
\usepackage[affil-sl]{authblk}
\usepackage[english,activeacute]{babel}
\usepackage{inputenc}
\usepackage[shortlabels]{enumitem}
\usepackage{orcidlink}

\newtheorem{theorem}{Theorem}
\newtheorem{lemma}{Lemma}[section]

\newtheorem{definition}{Definition}[section]

\newtheorem{remark}{Remark}[section]
\newtheorem{example}{Example}
\newcounter{Th-Alfa}

\newcommand{\CC}{\mathds{C}}

\newcommand{\PP}{\mathds{P}}

\newcommand{\RR}{\mathds{R}}
\newcommand{\ZZ}{\mathds{Z}}
\newcommand{\ZZp}{\mathds{Z}_+}
\newcommand{\dsty}{\displaystyle}

\newcommand{\supp}{\mathop{\rm supp}}

\newcommand{\inter}[1]{#1\strut^{\mathrm{o}}\!}
\newcommand{\ch}[1]{\mathbf{Co}\!\left(#1 \right)} 

\def\bbuildrel#1_#2^#3{\mathrel{
 \mathop{\kern 0pt#1}\limits_{#2}^{#3}}}
\def\bbbuildrel#1_#2{\mathrel{
 \mathop{\kern 0pt#1}\limits_{#2}}}

\newcommand{\funD}[4]{\begin{cases} #1 , & \hbox{if } \; #2;\\ & \\
										#3, & \hbox{if } \; #4 \end{cases}}
\def\bbuildrel#1_#2^#3{\mathrel{
 \mathop{\kern 0pt#1}\limits_{#2}^{#3}}}
\def\bbbuildrel#1_#2{\mathrel{
 \mathop{\kern 0pt#1}\limits_{#2}}}

\newcommand{\Lcoef}[1]{\mathop{\rm DgCo}\!\left(#1\right)}

\newcommand{\sob}{{\mathsf{s}}}
\newcommand{\Ip}[2]{\langle #1,#2\rangle}
\newcommand{\IpS}[2]{\Ip{#1}{#2}_{\sob}}

\title{\bf Electrostatic models for zeros of  Laguerre-Sobolev polynomials}
\author[1]{Abel D\'{\i}az-Gonz\'{a}lez\,\orcidlink{0000-0001-6226-1925}\thanks{abdiazgo@math.uc3m.es}}
\author[2]{H\'{e}ctor Pijeira-Cabrera\,\orcidlink{0000-0001-5141-9395}\thanks{hpijeira@math.uc3m.es}}
\author[2]{Javier Quintero-Roba\,\orcidlink{0000-0003-0536-3102}\thanks{jaquinte@math.uc3m.es}}
\affil[1]{Vanderbilt University, Nashville, Tennessee 37240, USA.}
\affil[2]{Departamento de Matem\'{a}ticas, Universidad Carlos III de Madrid,  \linebreak  Legan\'{e}s 28911,  Madrid,  Spain.}

\date{}

\begin{document}
\maketitle

\begin{abstract}
Let  {$\{S_n\}_{n\geqslant 0}$} be the sequence of orthogonal polynomials  with respect to the  Laguerre-Sobolev  inner product
$$
\langle f,g\rangle_S =\!\int_{0}^{+\infty}\! f(x) g(x)x^{\alpha}e^{-x}dx+\sum_{j=1}^{N}\sum_{k=0}^{d_j}\lambda_{j,k} f^{(k)}(c_j)g^{(k)}(c_j),
$$
where $\lambda_{j,k}\geqslant  0$, $\alpha >-1$ and   $c_i \in (-\infty, 0)$ for $i=1,2,\dots,N$.  We provide a formula that relates the Laguerre-Sobolev  polynomials $S_n$ to the standard Laguerre orthogonal polynomials. We find the ladder operators for the polynomial sequence $\{S_n\}_{n\geqslant 0}$ and a second-order differential equation with polynomial coefficients for $\{S_n\}_{n\geqslant 0}$.  We establish a sufficient condition for an electrostatic model of the zeros of orthogonal Laguerre-Sobolev polynomials. Some examples are given where this condition is either  satisfied or not.

\bigskip

 

\end{abstract}



\section{Introduction}

The  monic Laguerre polynomials {$\left\lbrace L_{n}^{\alpha} \right\rbrace_{n\geqslant 0}$}, $\alpha >-1$ (see \cite[Ch. 5]{Szego75})  {is}  a family of classical orthogonal polynomials defined by the orthogonality relations
$$
\langle L_n^\alpha, x^k \rangle_{\alpha}:=\int_0^{+\infty} L_n^\alpha(x) x^k d{\mu^{\alpha}}(x)=0, \quad \text { for } k=0,1, \ldots, n-1,
$$
where $d{\mu^{\alpha}}(x)= x^{\alpha}e^{-x}{dx}$, $\alpha >-1$. For the reader's convenience, we review some relevant notions and properties without proofs, which makes our exposition self-contained.  From \cite[(5.1.8), (5.1.6), (5.1.1), (5.1.7), (5.1.14) and (5.1.10)]{Szego75},  we have
\begin{align}\nonumber
 L_{n}^{\alpha}(x)& = (-1)^n n!\sum_{k=0}^{n}
\binom{n+\alpha }{n-k}\frac{(-x)^{k}}{k!}, \\ \nonumber  
(x-\beta_{n} )L_{n}^{\alpha }(x) & =L_{n+1}^{\alpha }(x)+\gamma_{n}L_{n-1}^{\alpha }(x),\quad n\geq 1, \\\label{Lag3TRR}
h^{\alpha}_n & =  \|L_{n}^{\alpha }\|_{\alpha }^{2}=\int_0^{+\infty} \left(L_n^\alpha(x)\right)^2 d{\mu^{\alpha}}(x)=n!\Gamma (n+\alpha +1),
\end{align}
where $\dsty  L_{0}^{\alpha }(x) =1$, $\dsty  L_{1}^{\alpha }(x)=x-\left( \alpha +1\right)$, $\dsty \beta_{n} =2n+\alpha +1,\quad \gamma_{n}   = n\left( n+\alpha \right)$, and  $\Gamma$ denotes the Gamma function
\begin{align*}
    \Gamma(z)=\int_{0}^\infty t^{z-1}e^{-t}dt,\quad \Re(z)>0.
\end{align*}

Let $\mathcal{I}$ be  the identity operator, and define the two ladder Laguerre differential operators on  the linear space of polynomials as

\begin{equation}\label{LagLadderOpDef}
\begin{aligned}
\widetilde{\mathcal{L}}_{n}^{\downarrow} & :=\frac{x}{\gamma_{n}} \;\frac{d}{dx}-\frac{n}{\gamma_{n}}\; \mathcal{I} & \text{(Lowering Laguerre differential operator)},\\
 \widetilde{\mathcal{L}}_{n}^{\uparrow} &:= -x \; \frac{d}{dx} + (x-n-\alpha)\; \mathcal{I}  & \text{(Raising Laguerre differential operator)} .
\end{aligned}
\end{equation}

From \cite[(5.1.14)]{Szego75}, for all $n \geqslant 1$, we have
\begin{equation}\label{LagLadderOp}
\widetilde{\mathcal{L}}_{n}^{\downarrow}\left[  L_{n}^{\alpha }(x)\right] = L_{n-1}^{\alpha }(x)\quad
\text{and} \quad \widetilde{\mathcal{L}}_{n}^{\uparrow}\left[  L_{n-1}^{\alpha }(x)\right] = L_{n}^{\alpha }(x).
\end{equation}%

Classical orthogonal polynomials (Jacobi, Laguerre, and Hermite) satisfy a second order differential equation with polynomial coefficients. Based on this fact, Stieltjes  gave  a very interesting interpretation of the  zeros of the classical orthogonal polynomials, as a solution to  an electrostatic equilibrium problem of $n$ movable unit charges in the presence of a logarithmic potential (see \cite[Sec.  3]{VanAss93} and \cite[Sec.  2]{ValAss95}). This is one of the main reasons for the importance of such polynomials for applications to boundary value problems in classical physics and quantum mechanics.

In the Laguerre case, $Y=L_{n}^{\alpha }(x)$ is a solution of the differential equation
\begin{equation}\label{Lag-DiffEqn}
x Y^{\prime \prime} +(\alpha+1-x) Y^{\prime} +n Y=0   \qquad \text{\cite[Ch. V, (2.20)]{Chi78}. }
\end{equation}

We will first  analyze the Stieltjes  electrostatic interpretation for the zeros of the Laguerre polynomials. Let us consider a system of $n$ unit charges distributed at points $\omega_1, \omega_2, \ldots, \omega_n$  in  $(0,+\infty)$ and  add one fixed positive charge of mass $(\alpha + 1)/2$ at $0$. In addition, we also consider  the following potential, which describes the interaction between the charges in the presence of an external field:
\begin{align}\label{LagEnergy}
\mathbf{E}\left(\omega_1, \omega_2, \ldots, \omega_n\right)=&\sum_{1 \leqslant i<j \leqslant n} \log \frac{1}{\left|\omega_i-\omega_j\right|}  + \frac{\alpha + 1}{2} \sum_{j=1}^n \log \frac{1}{\omega_j}+ \frac{1}{2} \sum_{j=1}^n \omega_j.
\end{align}

It is obvious,  that the minimum of \eqref{LagEnergy} gives the electrostatic equilibrium   of the system. The points $x_1,  \ldots, x_n$, where this minimum is achieved, are the places where the charges will settle down.   Therefore, all the $x_j$ are different from each other as well as from $0$ and $+\infty$.

 From the necessary condition for the existence of minimum  $\frac{\partial E_t}{\partial \omega_j}=0$ ($1\leqslant j \leqslant n$), it follows that the polynomial $Y=P_n(x)=\prod_{j=1}^{n}(x-x_j)$ satisfies the differential equation \eqref{Lag-DiffEqn}, which is the differential equation for the monic Laguerre polynomial, and then $P_n(x)=L_n^{\alpha}(x)$.

 On the other hand, the Hessian matrix of this energy function at $\overline{x}=(x_1,x_2,\dots,x_n)$ is given by

\begin{equation}\label{Hessian_M-class}
\nabla_{\overline{\omega}\,\overline{\omega}}^2 \mathbf{E}(\overline{x}) = \begin{cases}
\dsty \frac{\partial^2 \mathbf{E}}{\partial \omega_k\partial \omega_j}(\overline{x})=-\frac{1}{(x_{k}-x_{j})^{2}}, & \text{if } \;k\neq j,\\
\dsty \frac{\partial^2 \mathbf{E}}{\partial \omega_k^2}(\overline{x})= \sum_{{{i=1} \atop {i \neq k}}}^{n} \frac{1}{(x_{k}-x_{i})^2}+\frac{\alpha+1}{2x_k^2}, & \text{if } \;k = j.
\end{cases}
\end{equation}

 Since \eqref{Hessian_M-class} is a symmetric real matrix, its eigenvalues are real. Therefore, using Gershgorin's Theorem \cite[Th. 6.1.1]{Horn90}, the eigenvalues $\lambda$ of the Hessian at $\overline{x}$ satisfy
\begin{align*}
    \left|\lambda-\sum_{{{i=1} \atop {i \neq k}}}^{n} \frac{1}{(x_{k}-x_{i})^2}-\frac{\alpha+1}{2x_k^2}\right|\leq \sum_{{{i=1} \atop {i \neq k}}}^{n} \frac{1}{(x_{k}-x_{i})^2},
\end{align*}
for some  $k=1,2,\dots,n$. Then, we have  $\dsty     \lambda\geq \frac{\alpha+1}{2x_k^2}>0.$

Consequently, \eqref{Hessian_M-class} is positive definite, which implies that \eqref{LagEnergy} is a strictly convex function. Since $\nabla \mathbf{E}(\overline{x})=0$, we conclude that $\overline{x}$ is the only global minimum of \eqref{LagEnergy}. The methods used in this proof can be found in \cite[\S 2.3]{Ism00-A}, \cite{MarMarMar07} or \cite{OrGa10}. In conclusion, the global minimum of \eqref{LagEnergy} is reached when each of the $n$ charges is located on a zero of the $n$th Laguerre polynomial  $L_n^{\alpha}(x)$.

In \cite{MarMarMar07} you will find a survey of the results achieved up to fifteen years ago, on the electrostatic interpretation of the zeros of some well-known families of polynomials.  For more recent results, we refer the reader to the introductions of   \cite{Duenas11,HMP-PAMS14,PiQuiMi23}.

Let   $N,d_j \in \ZZ_+$, $\lambda_{j,k} \geqslant  0$,  for $j=1,\dots, N$, and $k=0,1,\dots,d_j$, and let the set $\{c_1,\dots,c_N\}\subset \RR \!\setminus\![0,\infty)$, where $c_i \neq c_j$ if $i \neq j$ and $I_+=\{(j,k): \lambda_{j,k}>0\}$.  Denote by $\PP$ the linear space of all polynomials with real coefficients. On $\PP$, we consider the following Sobolev-type inner product, which we will call \emph{Laguerre-Sobolev inner product}
\begin{align}\nonumber
	\IpS{f}{g}&= \Ip{f}{g}_{\alpha}+\sum_{j=1}^{N}\sum_{k=0}^{d_j}\lambda_{j,k} f^{(k)}(c_j)g^{(k)}(c_j)\\ \label{GeneralSIP}
	&=\!{\int_{0}^{+\infty}}\!f(x) g(x)d\mu_{\alpha}(x)+\sum_{(j,k)\in I_+}\!\!\!\!\lambda_{j,k} f^{(k)}(c_j)g^{(k)}(c_j),
\end{align}	
where $f^{(k)}$ denotes the $k$th derivative of  the polynomial $f$. Without loss of generality, we also assume $\{(j,d_j)\}_{j=1}^N\subset I_+$ and $d_1\leqslant d_2\leqslant \cdots \leqslant d_N$.  For $n \in \ZZp$ we shall denote by $S_n$ the monic polynomial of lowest degree satisfying
\begin{equation}\label{Sobolev-Orth}
\IpS{S_n}{x^k}= 0, \quad  \text{for } \; k=0,1,\dots,n-1.
\end{equation}
 It is easy to see that for all  $n\geqslant  0$, there exists such a unique polynomial $S_n$ of degree $n$. In fact, it is deduced by solving a homogeneous linear system with $n$ equations and $n+1$ unknowns. Uniqueness follows from the minimality of the degree for the polynomial solution. The polynomial $S_n$ is called the $n$th monic Laguerre-Sobolev orthogonal polynomial, for brevity,\emph{ $n$th monic Laguerre-Sobolev orthogonal polynomial}.

It is known that, in general, the properties of classical Laguerre polynomials differ from those of Laguerre-Sobolev polynomials.  In particular, unlike Laguerre polynomials, the zeros of Laguerre-Sobolev polynomials can be complex or, if they are real, they can lie outside $[0,\infty)$. For example, given the inner product
$$
   \IpS{f}{g}=  \int_{0}^{+\infty} f(x) g(x)e^{-x}dx+f'(-2)g'(-2),
$$
  the corresponding second-degree monic Laguerre-Sobolev polynomial is $S_2(z)=z^2-2$,  whose zeros are $\pm\sqrt 2$. Note that ${-\sqrt{2}}  \not \in [0,+\infty)$.

The aim of this paper is to provide an electrostatic model for the zeros of the Laguerre-Sobolev polynomials, following an approach based in the works \cite{HMP-PAMS14,PiQuiMi23}, the  main result of  M.E.H. Ismail in \cite{Ism00-A,Ism00-B}, and the original ideas of Stieltjes in \cite{Sti-1885-1,Sti-1885-2}.
Our results extend those achieved in \cite{DuenasMarcellan,HMP-PAMS14,Molano13} to a more general context and complement those obtained in \cite{PiQuiMi23} for the Jacobi-Sobolev case.

In the next section, we review a connection formula that allows us to express the polynomial $S_n$, as a linear combination of the Laguerre polynomials, whose coefficients are rational functions. Sections \ref{LadderOp} and \ref{Sec-DiffEqn} deal with  the extension of the ladder (raising and lowering) differential operators  \eqref{LagLadderOpDef}, and the  second-order differential equation with polynomial coefficients \eqref{Lag-DiffEqn}, for the Laguerre-Sobolev polynomials.

In the last section, we give sufficient conditions for   an electrostatic model for the distribution of the zeros of {$\{S_{n}\}_{n\geqslant 0}$}. Such model is expressed   as the logarithmic potential interaction of unit positive charges in the presence of an external field. Some examples are given where this condition is either satisfied or not.

\section{Auxiliary results}\

In this section, for the reader's convenience, we repeat some results from \cite[(21)-(22)]{DiHerPi23} and \cite[Sec. 2]{PiQuiMi23} without proofs, which makes our exposition self-contained.


 We first recall the well-known Christoffel-Darboux formula for  $K_{n}(x,y)$, the kernel polynomials associated with  $\{L^{\alpha}_n\}_{n\geqslant  0}$ (see \cite[(5.1.8) and (5.1.11)]{Szego75}).
 \begin{equation}\label{Kernel-n}
K_{n-1}(x,y)=\sum_{k=0}^{n-1}\frac{L^{\alpha}_{k}(x)L^{\alpha}_{k}(y)}{h^{\alpha}_k}= \funD{\dsty \frac{L^{\alpha}_{n}(x)L^{\alpha}_{n-1}(y)-L^{\alpha}_{n}(y)L^{\alpha}_{n-1}(x)}{h^{\alpha}_{n-1}\, (x-y)}}{x\neq y}{\dsty \frac{\left(L^{\alpha}_{n}(x)\right)^{\prime}  L^{\alpha}_{n-1}(x)-L^{\alpha}_{n}(x)\left(L^{\alpha}_{n-1}(x)\right)^{\prime}}{h^{\alpha}_{n-1}}}{x=y.}
\end{equation}

  We will denote by  $\dsty K_{n}^{(j,k)}\left( x,y\right) =\frac{\partial ^{j+k}K_{n}\left( x,y\right) }{\partial x ^{j}\partial y^{k}}$ the partial derivatives of  \eqref{Kernel-n}. Then, from the Christoffel-Darboux formula  and Leibniz's rule, it is not difficult to verify that
\begin{align}\nonumber
K_{n-1}^{(0,k)}(x,y)=& \sum_{i=0}^{n-1}\frac{L^{\alpha}_{i}(x)\left(L^{\alpha}_i(y)\right)^{(k)}}{h^{\alpha}_{i}} \\ =& \frac{k!\left(
T_{k}(x,y;L^{\alpha}_{n-1})L^{\alpha}_{n}(x)-T_{k}(x,y;L^{\alpha}_{n})L^{\alpha}_{n-1}(x)\right)}{h^{\alpha}_{n-1}\;(x-y)^{k+1}}
, \label{CDformuladj}
\end{align}
where $T_{k}(x,y;f)=\sum_{\nu=0}^{k} \frac{f^{(\nu)}(y)}{\nu !}(x- y)^\nu$   is  the Taylor polynomial of degree $k$ of $f$ centered at $y$. Note that \eqref{Kernel-n} is a particular case of \eqref{CDformuladj} when   $k=0$.

From \eqref{GeneralSIP}, if $i<n$
\begin{equation}\label{FourierCoeff}
  \Ip{S_n}{L^{\alpha}_i}_{\alpha }=  \IpS{S_n}{L^{\alpha}_i}-\sum_{(j,k)\in I_+}\lambda_{j,k} S_n^{(k)}(c_j)\left(L^{\alpha}_i\right)^{(k)}\!\!(c_j)=-\sum_{(j,k)\in I_+}\lambda_{j,k} S_n^{(k)}(c_j)\left(L^{\alpha}_i\right)^{(k)}\!\!(c_j).
\end{equation}
Therefore, from   the Fourier expansion of $S_{n}$ in terms of the basis {$\left\{ L_{n}^{\alpha}\right\} _{n\geqslant 0}$} and using \eqref{FourierCoeff}, we get
\begin{align}
S_{n}(x)&=L^{\alpha}_n(x)+{\sum_{i=0}^{n-1}\Ip{S_n}{L^{\alpha}_i}_{\alpha}}\frac{L^{\alpha}_{i}(x)}{h^{\alpha}_{i}}= L^{\alpha}_n(x)-\sum_{(j,k)\in I_+}\lambda_{j,k} S_n^{(k)}(c_j)\sum_{i=0}^{n-1}\frac{L^{\alpha}_{i}(x)\left(L^{\alpha}_i\right)^{(k)}\!\!(c_j)}{h^{\alpha}_{i}} \nonumber
\\
&=L^{\alpha}_{n}(x)-\sum_{(j,k)\in I_+}\lambda_{j,k} S_{n}^{(k)}(c_{j})K_{n-1}^{(0,k)}(x,c_j).\label{FConexFINAL}
\end{align}

Now, replacing \eqref{CDformuladj} in \eqref{FConexFINAL}, we have  the \emph{connection formula}
\begin{align}\label{FConexPPAL}
S_{n}(x)& = F_{1,n}(x)L^{\alpha}_{n}(x) + G_{1,n}(x)L^{\alpha}_{n-1}(x), \\ \nonumber
\text{where} \quad  F_{1,n}(x) & = 1-\sum_{(j,k)\in I_+} \frac{\lambda _{j,k}k!\,S_{n}^{(k)}(c_{j})}{h^{\alpha}_{n-1}}\frac{T_{k}(x,c_{j};L^{\alpha}_{n-1})}{(x-c_{j})^{k+1}} \\ \nonumber
\text{and} \quad  G_{1,n}(x)&=\sum_{(j,k)\in I_+} \frac{\lambda _{j,k}k!\,S_{n}^{(k)}(c_{j})}{h^{\alpha}_{n-1}}\frac{T_{k}(x,c_{j};L^{\alpha}_{n})}{(x-c_{j})^{k+1}}.
\end{align}

Deriving equation  \eqref{FConexFINAL} $\ell$-times and evaluating then at $x=c_{i}$ for each ordered pair $(i,\ell)\in I_+$ we obtain the following system of  $d^*=\#(I_+)$  linear equations and $d^*$ unknowns $S^{(k)}_n(c_{j})$, where the symbol $\#(A)$ denote  the cardinality of a given set $A$.
\begin{align}\label{eqsystem1}
\left(L^{\alpha}_{n}(c_{i})\right)^{(\ell)}=\left(1+\lambda_{i,\ell}K_{n-1}^{(\ell,\ell)}(c_i,c_i)\right)S^{(\ell)}_n(c_{i})\ +\!\!\!\sum_{\substack{(j,k)\in I_+\\ (j,k)\neq (i,\ell)}}\!\!\!\!\!\lambda_{j,k} K_{n-1}^{(\ell,k)}(c_{i},c_j)S^{(k)}_n(c_{j}).
\end{align}

The remainder of this section is devoted to {proving that system \eqref{eqsystem1} has a unique solution}. The following lemma is essential to achieve this goal.

\begin{lemma}{\cite[Lemma 1]{PiQuiMi23}}\label{LemmaFullRank}
Let  $I\subset \mathbb{R}\times \mathbb{Z}_+$ be a (finite) set of $d^*$ pairs. Denote $\{c_j\}_{j=1}^N=\pi_1(I)$ where $\pi_1$ is the projection function over the first coordinate, i.e., $\pi_1(x,y)=x$, $d_j=\max\{\nu_i: (c_j,\nu_i)\in I\}$ and $d=\sum_{j=1}^N (d_j+1)$. Let   $P_k$ be an arbitrary polynomial of degree $k$ for $0\leqslant k\leqslant n-1$. Then,  for all $n\geqslant  d$, the $ d^*\!\times \!n$ matrix
 \begin{align*}
 \mathfrak{A}^*=\left(P^{(\nu)}_{k-1}(c)\right)_{(c,\nu)\in I,k=1,2,\dots,n}, \quad \text{ has full rank $d^*$.}
\end{align*}
\end{lemma}

Now we can rewrite  \eqref{eqsystem1} in the matrix form
\begin{align}\label{solve1}
 \mathfrak{P}_{n}( \mathfrak{C})=( \mathfrak{I}_{d^{*}}+ \mathfrak{K}_{n-1}( \mathfrak{C}, \mathfrak{C}) \mathfrak{L}) \mathfrak{S}_{n}( \mathfrak{C}),
\end{align}
where
\begin{description}
  \item[$ \mathfrak{I}_{d^{*}}$] is the identity matrix of order $d^*$,
    \item[$ \mathfrak{L}$ ]  is the $d^*\!\!\times\! d^*$-diagonal matrix with diagonal entries $\lambda_{j,k}$, $(j,k)\in I_+ $,
  \item[$ \mathfrak{C}$] is the column vector $\dsty  \mathfrak{C}=(\underbrace{c_{1},\dots,c_1}_{d^*_1 \text{-times}},\underbrace{c_2,\dots,c_2}_{d^*_2 \text{-times}},\dots,
\underbrace{c_N,\dots,c_N}_{d^*_N \text{-times}})^\intercal$,
  \item[$ \mathfrak{P}_{n}( \mathfrak{C})$] \!and $ \mathfrak{S}_{n}( \mathfrak{C})$  are column vectors with entries $(L^{\alpha}_n)^{(k)}(c_j)$, and $S^{(k)}_n(c_j)$, $(j,k)\in I_+$ respectively, and
  \item[$ \mathfrak{K}_{n-1}( \mathfrak{C}, \mathfrak{C})$]  is a $d^*\times d^*$ matrix  having $K_{n-1}^{(\ell,k)}(c_{i},c_{j})$ in the $(i,\ell)$th row and the $(j,k)$th column for  $(i,\ell),(j,k)\in I_+$.
\end{description}

Clearly,  we can write $\dsty  \mathfrak{K}_{n-1}( \mathfrak{C}, \mathfrak{C})= \mathfrak{F} \mathfrak{F}^\intercal,$ where {$ \mathfrak{F}=\left(\frac{\left(L^{\alpha}_{\nu-1}\right)^{(k)}(c_j)}{\sqrt{h^{\alpha}_{\nu-1}}}\right)_{(j,k)\in I_+, \nu=1,\dots,n,}$} is a matrix  of order $d^*\!\!\times \! n$ and  full rank  for all $n\geqslant  d$, according to   Lemma \ref{LemmaFullRank}.

Then, the matrix $ \mathfrak{K}_{n-1}( \mathfrak{C}, \mathfrak{C})$ is a $d^*\!\!\times\! d^*$ positive definite matrix for all $n\geqslant  d$, see \cite[Th. 7.2.7(c)]{Horn90}. Since $ \mathfrak{L}$ is a diagonal matrix with positives entries, it follows that $ \mathfrak{L}^{-1}+ \mathfrak{K}_{n-1}( \mathfrak{C}, \mathfrak{C})$ is also a positive definite matrix and consequently $ \mathfrak{I}_{d^{*}}+ \mathfrak{K}_{n-1}( \mathfrak{C}, \mathfrak{C}) \mathfrak{L}=\left( \mathfrak{L}^{-1}+ \mathfrak{K}_{n-1}( \mathfrak{C}, \mathfrak{C})\right)\! \mathfrak{L}$ is non singular. Then, the linear system  \eqref{solve1} has the unique solution
\begin{align*}
 \mathfrak{S}_{n}( \mathfrak{C})=( \mathfrak{I}_{d^{*}}+ \mathfrak{K}_{n-1}( \mathfrak{C}, \mathfrak{C}) \mathfrak{L})^{-1} \mathfrak{P}_{n}( \mathfrak{C}).
\end{align*}
Using this notation we can rewrite  \eqref{FConexFINAL} in the compact form
\begin{equation*}
S_{n}(x)={L^{\alpha}_{n}(x)}- \mathfrak{K}_{n-1}(x, \mathfrak{C})\, \mathfrak{L} \, \mathfrak{S}%
_{n}( \mathfrak{C}),
\end{equation*}
where $ \mathfrak{K}_{n-1}(x, \mathfrak{C})$ is a row vector with entries $K_{n-1}^{(0,k)}(x,c_{j})$, for $(j,k)\in I_+$.


\section{Differential operators}\label{LadderOp}

Let
\begin{align}
   \rho(x)&: =\prod_{j=1}^{N} \left( x-c_j\right)^{d_j+1} \quad \text{and} \label{rho-def}\\
\rho _{j,k}(x)&:=\frac{\rho(x)}{(x-c_j)^{k+1}}=(x-c_{j})^{d_j-k}\prod_{\substack{i=1 \\ i\neq j}}^N(x-c_{i})^{d_i+1}, \quad \text{for every $(j,k)\in I_+$.}\label{LagrhoMk}
\end{align}

 If $p$ is a polynomial of degree $n$ (i.e., $p(x)=a_nx^n+\cdots+a_0$), we denote by $\Lcoef{p}$ the ordered pair whose first component is the degree of $p$ and   second component is  its  leading coefficient, i.e.,
$$\Lcoef{p}=\left(n,\mathop{Coef}_{n}(p)\right)$$
where $\dsty \mathop{Coef}_{k}(p)=a_k$ is the coefficient of $x^k$ in  $p$.



\begin{lemma}
\label{lemma31}  The sequences of polynomials $\{S_{n}\}_{n\geq 0}$ and $\{L^{\alpha}_{n}\}_{n\geq 0}$, satisfy
\begin{eqnarray}
\rho(x)S_{n}(x) &=& F_{2,n}(x) \;  L^{\alpha}_{n}(x)+G_{2,n}(x) \;L^{\alpha}_{n-1}(x),
\label{LemConnForm1} \\
x\left( \rho (x)S_{n}(x)\right) ^{\prime }
&=& F_{3,n}(x)  L^{\alpha}_{n}(x)+ G_{3,n}(x)  L^{\alpha}_{n-1}(x),  \label{LemConnForm1-Dx}
\end{eqnarray}%
where%
\begin{align*}
F_{2,n}(x)  =&\rho(x)F_{1,n}(x)   =\rho(x)-\sum_{(j,k)\in I_+}\left(\frac{k! \lambda
_{j,k}S_{n}^{(k)}(c_{j})}{h^{\alpha}_{n-1}}\,T_{k}\left(x,c_{j};L^{\alpha}_{n-1}\right)\right) \rho _{j,k}(x), \\
G_{2,n}(x)  =&\rho(x)G_{1,n}(x)   = \sum_{(j,k)\in I_+}\left(\frac{%
k! \lambda _{j,k}S_{n}^{(k)}(c_{j})}{h^{\alpha}_{n-1}}%
\,T_{k}\left(x,c_{j};L^{\alpha}_{n}\right)\right) \rho _{j,k}(x), \\
F_{3,n}(x) = &  x F_{2,n}'(x)+ n F_{2,n}(x)- G_{2,n}(x), \\
G_{3,n}(x) = &  x G_{2,n}'(x) + \gamma_n F_{2,n}(x)-(n+\alpha-x)G_{2,n}(x),
\end{align*}
where $F_{2,n}$, $G_{2,n}$, $F_{3,n}$, and $G_{3,n}$ are polynomials  with the following degrees and leading coefficients
\begin{align*}
    \Lcoef{F_{2,n}}&=(d,1),\\
    \Lcoef{G_{2,n}}&=(d-1,\sigma_n),\\
    \Lcoef{F_{3,n}}&=(d,d+n),\\
    \Lcoef{G_{3,n}}&=(d,\gamma_n+\sigma_n);
\end{align*}
and
\begin{align*}
\sigma_n=\frac{1}{h^{\alpha}_{n-1}}\sum_{(j,k)\in I_+}\lambda _{j,k}S_{n}^{(k)}(c_{j})\left(L^{\alpha}_n\right)^{(k)}\!(c_j)>0.
\end{align*}
\end{lemma}



\begin{proof}
From  \eqref{FConexPPAL}-\eqref{LagrhoMk}, equation \eqref{LemConnForm1} is immediate. To prove \eqref{LemConnForm1-Dx}, we can take derivatives with respect to $x$  at both sides  of \eqref{LemConnForm1} and then multiply by $x$%
\begin{align*}
\begin{split}
x \left( \rho (x)S_{n}(x)\right) ^{\prime }=& x F_{2,n}' L_{n}^{\alpha}(x)+xF_{2,n} \left(L_{n}^{\alpha}(x)\right)'\\
&+ x G_{2,n}' L_{n-1}^{\alpha}(x)+xG_{2,n} \left(L_{n-1}^{\alpha}(x)\right)'.
\end{split}
\end{align*}%
Using \eqref{LagLadderOp},   we obtain \eqref{LemConnForm1-Dx} as follows:
\begin{align*}
x \left(\rho(x)S_{n}(x)\right)^{\prime }=&\left[x F_{2,n}'(x)+ n F_{2,n}(x)- G_{2,n}(x)\right]L^{\alpha}_{n}(x)
\\
&+\left[x G_{2,n}'(x) + \gamma_n F_{2,n}(x)-(n+\alpha-x)G_{2,n}(x)\right] L_{n-1}^{\alpha}(x).
\end{align*}%
From the expressions for $F_{2,n}$, we directly get that $F_{2,n}$ is monic with degree $d$. It is also straightforward that $deg(G_{2,n})\leq d-1$. Let us continue by finding the coefficient of $G_{2,n}$ of the power $x^{d-1}$
\begin{align*}
   G_{2,n}(x)  =& \sum_{(j,k)\in I_+}\left(\frac{k! \lambda _{j,k}S_{n}^{(k)}(c_{j})}{h^{\alpha}_{n-1}}
\,T_{k}\left(x,c_{j};L^{\alpha}_{n}\right)\right) \rho _{j,k}(x)\\
=& \sum_{(j,k)\in I_+}\left(\frac{k! \lambda _{j,k}S_{n}^{(k)}(c_{j})}{h^{\alpha}_{n-1}}
\,\sum_{\nu=0}^{k} \frac{\left(L^{\alpha}_n\right)^{(\nu)}(c_j)}{\nu !}(x- c_j)^\nu\right) \rho _{j,k}(x).
\end{align*}
Since $deg(\rho _{j,k}(x))=d-1-k$, the coefficient of $G_{2,n}$ of the power $x^{d-1}$ is given by
\begin{align*}
   \sum_{(j,k)\in I_+}\left(\frac{k! \lambda _{j,k}S_{n}^{(k)}(c_{j})}{h^{\alpha}_{n-1}}
\,\frac{\left(L^{\alpha}_n\right)^{(k)}(c_j)}{k!}\right)=\frac{1}{h^{\alpha}_{n-1}}\sum_{(j,k)\in I_+}\lambda _{j,k}S_{n}^{(k)}(c_{j})\left(L^{\alpha}_n\right)^{(k)}\!(c_j)=\sigma_n.
\end{align*}
Notice that $\sigma_n$ is positive. Otherwise, we get
{
\begin{align*}
    0&=\IpS{S_n}{S_n-L^{\alpha}_n}= \Ip{S_n}{S_n-L^{\alpha}_n}_{\alpha}+\sum_{(j,k)\in I_+}\lambda _{j,k}S_{n}^{(k)}(c_{j})\left(S_n-L^{\alpha}_n\right)^{(k)}\!(c_j) \\
    &=\|S_n\|^2_{\alpha}-\Ip{S_n}{L^{\alpha}_n}_{\alpha}+\sum_{(j,k)\in I_+}\lambda _{j,k}S_{n}^{(k)}(c_{j})S_n^{(k)}\!(c_j)-h^{\alpha}_{n-1}\sigma_n \\
    &=\|S_n\|^2_{\alpha}-\|L^{\alpha}_n\|^2_{\alpha}+\sum_{(j,k)\in I_+}\lambda _{j,k}\left(S_{n}^{(k)}(c_{j})\right)^2-h^{\alpha}_{n-1}\sigma_n \\
    &\geq \|S_n\|^2_{\alpha}-\|L^{\alpha}_n\|^2_{\alpha},
\end{align*}
}
which contradicts the minimality of the norm of the Laguerre polynomials. Then, $G_{2,n}$ has degree $=d-1$ and a positive leading coefficient $\sigma_n$. Finally, the degrees and leading coefficients of $F_{3,n}$ and $G_{3,n}$ follows directly from the degrees and leading coefficients of $F_{2,n}$ and $G_{2,n}$.

\end{proof}



\begin{lemma}\label{lemma32}
 The sequences of monic polynomials $\left\{S_{n}\right\}_{n\geq 0}$ and
$\left\{L_{n}^{\alpha}\right\}_{n\geq 0}$ are also related by the equations
\begin{align}
\rho (x)S_{n-1}(x) &=V_{2,n}(x) L_{n}^{\alpha}(x)+W_{2,n}(x) L_{n-1}^{\alpha}(x),
\label{LemConnForm2} \\
x\left(\rho(x) S_{n-1}(x)\right)^{\prime }
&=V_{3,n}(x)L_{n}^{\alpha}(x)+W_{3,n}(x) L_{n-1}^{\alpha}(x),  \label{LemConnForm2-Dx}
\end{align}%
where%
$$
\begin{array}{lllll}
\dsty V_{2,n}(x)  & =  \dsty -\frac{G_{2,n-1}(x)}{\gamma_{n-1}},&\quad
\dsty W_{2,n}(x)  & = \dsty F_{2,n-1}(x)+G_{2,n-1}(x)\left(\frac{x-\beta_{n-1}}{\gamma_{n-1}}\right),\\
\dsty V_{3,n}(x)  & = \dsty -\frac{G_{3,n-1}(x)}{\gamma_{n-1}},&\quad
\dsty W_{3,n}(x)  & = \dsty F_{3,n-1}(x)+G_{3,n-1}(x)\left(\frac{x-\beta_{n-1}}{\gamma_{n-1}}\right),\\
\end{array}
$$
where $V_{2,n}$, $W_{2,n}$, $V_{3,n}$ and $W_{3,n}$ are polynomials with the following degrees and leading coefficients
\begin{align*}
    \Lcoef{V_{2,n}}&=(d-1,-\sigma_{n-1}/\gamma_{n-1}),\\
    \Lcoef{W_{2,n}}&=(d,1+\sigma_{n-1}/\gamma_{n-1}),\\
    \Lcoef{V_{3,n}}&=(d,-(1+\sigma_{n-1}/\gamma_{n-1})),\\
    \Lcoef{W_{3,n}}&=(d+1,1+\sigma_{n-1}/\gamma_{n-1}).
\end{align*}
\end{lemma}



\begin{proof}
The proof of \eqref{LemConnForm2}--\eqref{LemConnForm2-Dx} is a
straightforward consequence of Lemma \ref{lemma31} and the three term-recurrence relation \eqref{Lag3TRR}.
\end{proof}


\begin{lemma}\label{lemma33}
The monic orthogonal Laguerre polynomials {$\left\{L_{n}^{\alpha}\right\}_{n\geqslant 0}$}%
, can be expressed in terms of the monic Sobolev-type polynomials {$\left\{S_{n}\right\}_{n\geqslant 0}$} in the following way:
\begin{align}
L_{n}^{\alpha}(x) &=\frac{\rho (x)}{\Delta_n(x)}\left(W_{2,n}(x)S_{n}(x)-G_{2,n}(x)S_{n-1}(x)\right) ,  \label{Ln-desp-01} \\
L_{n-1}^{\alpha}(x) &=\frac{\rho (x)}{\Delta_n(x)}\left(F_{2,n}(x)S_{n-1}(x)-V_{2,n}(x)S_{n}(x)\right) ;  \label{Ln-desp-02}
\end{align}%
where%
\begin{equation} \label{Main-Determinant}
\Delta_n(x)=\det\begin{pmatrix}
F_{2,n}(x) & G_{2,n}(x) \\
V_{2,n}(x) & W_{2,n}(x)%
\end{pmatrix}%
=F_{2,n}(x)W_{2,n}(x)-V_{2,n}(x)G_{2,n}(x)
\end{equation}%
is a polynomial with a degree and leading coefficient
$$\Lcoef{\Delta_n}=(2d,1+\sigma_{n-1}/\gamma_{n-1}).$$
\end{lemma}



\begin{proof}
Note that equations \eqref{LemConnForm1} and \eqref{LemConnForm2} can be seen as  a system of two linear equations with two unknowns $L_{n}^{\alpha}(x)$ and $L_{n-1}^{\alpha}(x)$. Then, from Cramer's rule, we get \eqref{Ln-desp-01} and \eqref{Ln-desp-02}.

 The degree of $\Delta_n$ can be computed easily from the degrees of $F_{2,n}$, $G_{2,n}$, $V_{2,n}$, and  $W_{2,n}$ given in Lemmas \ref{lemma31} and \ref{lemma32}.

\end{proof}



 The following ladder equations follow from the last three   lemmas.
\begin{remark} \label{remark31} Since the zeros of $L_n^{\alpha}$ {lie} on $(0,\infty)$, from \eqref{Ln-desp-01} (or \eqref{Ln-desp-02}) we have that $\rho$ divides $\Delta_n$; this is,  there exists a polynomial $\delta_n$ such that $\Delta_n= \rho  \delta_n$. Hence, from \eqref{Main-Determinant} and  Lemma \ref{lemma31} we obtain
\begin{align*}
 \delta_n(x) & = \frac{\Delta_n(x)}{\rho(x)}=\frac{F_{2,n}(x)W_{2,n}(x)-V_{2,n}(x)G_{2,n}(x)}{\rho(x)}=F_{1,n}(x)W_{2,n}(x)-G_{1,n}(x)V_{2,n}(x). 
\end{align*}
In addition, from $\Delta_n= \rho \delta_n$ we obtain
$$\Lcoef{\delta_n}=(d,1+\sigma_{n-1}/\gamma_{n-1}).$$
\end{remark}



\begin{theorem}[Ladder equations]\label{lemma44} Under the above assumptions,  we have the following ladder equations.
\begin{align}
F_{4,n}(x)S_{n}(x)+G_{4,n}(x)S_{n}^{\prime}(x) & =S_{n-1}(x),
\label{laddereq1} \\
V_{4,n}(x)S_{n-1}(x)+W_{4,n}(x)S_{n-1}^{\prime}(x) & = S_{n}(x),
\label{laddereq2}
\end{align}%
where
\begin{align*}
 F_{4,n}(x)  & = \frac{q_{2,n}(x) }{q_{1,n}(x)},  \;    G_{4,n}(x)  = \frac{q_{0,n}(x)}{q_{1,n}(x)}\;
V_{4,n}(x)   = \frac{q_{3,n}(x)}{q_{4,n}(x)},\;   W_{4,n}(x)   = \frac{q_{0,n}(x)}{q_{4,n}(x)};\\
q_{0,n}(x)&= x\Delta_n(x),\quad  \Lcoef{q_{0,n}}=\left(2d+1,1+\frac{\sigma_{n-1}}{\gamma_{n-1}}\right); \\
q_{1,n}(x)&= G_{3,n}(x)F_{2,n}(x)-F_{3,n}(x)G_{2,n}(x),\quad \Lcoef{q_{1,n}}= (2d,\gamma_{n}+\sigma_{n}); \\
q_{2,n}(x) &= x\rho^{\prime}(x)\delta_n(x)+G_{3,n}(x)V_{2,n}(x)- F_{3,n}(x)W_{2,n}(x),\quad \Lcoef{q_{2,n}}= \left(2d,-n\left[1+\frac{\sigma_{n-1}}{\gamma_{n-1}}\right]\right); \\ 
q_{3,n}(x) & = x\rho^{\prime}(x)\delta_n(x) +V_{3,n}(x)G_{2,n}(x)-W_{3,n}(x)F_{2,n}(x),\quad \Lcoef{q_{3,n}}= \left(2d+1,-\left[1+\frac{\sigma_{n-1}}{\gamma_{n-1}}\right]\right);   \\
 \text{and }\;\;& \\ q_{4,n}(x)& = V_{3,n}(x)W_{2,n}(x)-W_{3,n}(x)V_{2,n}(x),  \quad \Lcoef{q_{4,n}}= \left(2d,-\left[1+\frac{\sigma_{n-1}}{\gamma_{n-1}}\right]\right).\\
\end{align*}
\end{theorem}



\begin{proof}
Replacing \eqref{Ln-desp-01} and \eqref{Ln-desp-02} in \eqref{LemConnForm1-Dx}
and \eqref{LemConnForm2-Dx}, the two ladder equations \eqref{laddereq1} and\eqref{laddereq2}  follow. Notice that
\begin{align*}
F_{4,n}(x)  & = \frac{x\left(\rho(x)\right)^{-1}\rho^{\prime}(x)\Delta_n(x)-\left(F_{3,n}(x)W_{2,n}(x)-G_{3,n}(x)V_{2,n}(x)\right)}{G_{3,n}(x)F_{2,n}(x)-F_{3,n}(x)G_{2,n}(x)},\\
G_{4,n}(x)   & = \frac{x\Delta_n(x)}{G_{3,n}(x)F_{2,n}(x)-F_{3,n}(x)G_{2,n}(x)},\\
V_{4,n}(x)   & = \frac{x\left(\rho(x)\right)^{-1}\rho^{\prime}(x)\Delta_n(x)-\left(W_{3,n}(x)F_{2,n}(x)-V_{3,n}(x)G_{2,n}(x)\right)}{V_{3,n}(x)W_{2,n}(x)-W_{3,n}(x)V_{2,n}(x)},\\
W_{4,n}(x)   & = \frac{x\Delta_n(x)}{V_{3,n}(x)W_{2,n}(x)-W_{3,n}(x)V_{2,n}(x)}.
\end{align*}
The computations for the degrees and leading coefficients are a direct consequence of the expressions obtained, Remark \ref{remark31}, and Lemmas \ref{lemma31},\ref{lemma32},\ref{lemma33}.

\end{proof}

In the previous theorem, the polynomials $q_{k,n}$ have been defined. Note that these polynomials are closely related to certain determinants. The following result summarizes some of its properties that will be of interest later. For brevity, we introduce the following notations
\begin{align*}
\Delta_{1,n}(x)&=G_{3,n}(x)F_{2,n}(x)-F_{3,n}(x)G_{2,n}(x); \\
\Delta_{2,n}(x) &=G_{3,n}(x)V_{2,n}(x)- F_{3,n}(x)W_{2,n}(x); \\
\Delta_{3,n}(x) & =G_{2,n}(x)V_{3,n}(x)-F_{2,n}(x)W_{3,n}(x).
\end{align*}

\begin{lemma}\label{LemmaDeterminant} Let $\dsty \rho_{N}(x) =\prod_{j=1}^{N} \left( x-c_j\right)$ and $\dsty \rho_{d-N}(x)=\prod_{j=1}^{N} \left( x-c_j\right)^{d_j}=\frac{\rho(x)}{\rho_{N}(x)}$. Then, the above polynomial determinants admit the following decompositions
\begin{align*}
 \Delta_{1,n}(x)&= \rho_{d-N}(x) \; \varphi_{1,n}(x), \\
 \Delta_{2,n}(x)&=  \rho_{d-N}(x)\; \varphi_{2,n}(x),\\
  \Delta_{3,n}(x)&=  \rho_{d-N}(x) \; \varphi_{3,n}(x),
\end{align*}
where
\begin{equation}\label{deglemaDetrminant}
\begin{aligned}
    \Lcoef{\Delta_{1,n}}&=(2d,\gamma_n+\sigma_n),\\
    \Lcoef{\Delta_{2,n}}&=(2d,-(d+n)(1+\sigma_{n-1}/\gamma_{n-1})),\\
     \Lcoef{\Delta_{3,n}}&=(2d+1,-(1+\sigma_{n-1}/\gamma_{n-1})),\\
     \Lcoef{\varphi_{1,n}}&=(d+N,\gamma_n+\sigma_n),\\
    \Lcoef{\varphi_{2,n}}&=(d+N,-(d+n)(1+\sigma_{n-1}/\gamma_{n-1})), \quad \text{and}\\
    \Lcoef{\varphi_{3,n}}&=(d+N+1,-(1+\sigma_{n-1}/\gamma_{n-1})).
\end{aligned}
\end{equation}
\end{lemma}

\begin{proof} Multiplying \eqref{LemConnForm1}  by $G_{3,n}$, \eqref{LemConnForm1-Dx} by $G_{2,n}$ and taking their difference, we have
\begin{align*}
\Delta_{1,n}(x) L^{\alpha}_n(x)  = &\; \rho(x) G_{3,n}(x) S_n(x)-x G_{2,n}(x)  \left(\rho^{\prime}(x) S_n(x)+\rho(x)   S_n^{\prime}(x) \right)\\
= & \; \rho_{d-N}(x) \Big( \rho_{N}(x)  [G_{3,n}(x) S_n(x)-xG_{2,n}(x)S_n^{\prime}(x)]\\
 & -xG_{2,n}(x)S_n(x)  \sum_{j=1}^{N}(d_j+1)\, \prod_{i\neq j}(x-c_i) \,\Big) .
\end{align*}
Since the zeros of $L_n^{\alpha}$ {lie} on $(0,\infty)$, we have that $\rho_{n-N}$ divides $\Delta_{1,n}$, i.e., there exists a polynomial  $\varphi_{1,n}$ such that $\Delta_{1,n} = \rho_{d-N}\varphi_{1,n}$. Using Lemma \ref{lemma31}, we obtain
\begin{align*}
    \Lcoef{\Delta_{1,n}}&=(2d,\gamma_n+\sigma_n),\\\
    \Lcoef{\varphi_{1,n}}&=(d+N,\gamma_n+\sigma_n).
\end{align*}
For the decomposition of $\Delta_{2,n}$ ($\Delta_{3,n}$), the procedure of the proof is analogous, using the linear system of \eqref{LemConnForm1-Dx}-\eqref{LemConnForm2} (\eqref{LemConnForm1}-\eqref{LemConnForm2-Dx}) and Lemmas \ref{lemma31}, \ref{lemma32} for the degrees and leading coefficients.

\end{proof}

\begin{definition}[Ladder Laguerre-Sobolev differential operators] Let $\mathcal{I}$ be  the identity operator. We define the two ladder differential operators on $\PP$ as
\begin{align*}
\mathcal{L}_{n}^{\downarrow} & := F_{4,n}(x) \mathcal{I}+G_{4,n}(x) \frac{d}{dx}\quad  \text{(Lowering Laguerre-Sobolev differential operator)},\\
 \mathcal{L}_{n}^{\uparrow} &:= V_{4,n}(x) \mathcal{I}+W_{4,n}(x) \frac{d}{dx} \quad \text{(Raising Laguerre-Sobolev differential operator)} .
\end{align*}
\end{definition}

\begin{remark} \label{Remark-Operators} Assuming in \eqref{GeneralSIP} that  $\lambda_{j,k}\equiv 0$ for all pair $(j,k)$, it is not difficult to verify that $\mathcal{L}_{n}^{\downarrow}$ and $\mathcal{L}_{n}^{\uparrow}$ simplify to the expressions given in \eqref{LagLadderOp}.
\end{remark}


Now,  we can rewrite the ladder equations \eqref{laddereq1} and \eqref{laddereq2} as%
\begin{align}
\mathcal{L}_{n}^{\downarrow} \left[S_{n}(x) \right]&=\left( F_{4,n}(x)\mathcal{I}+G_{4,n}(x)\frac{d}{dx}\right)S_{n}(x) =S_{n-1}(x),  \label{laddereq1rew} \\
\mathcal{L}_{n}^{\uparrow}\left[ S_{n-1}(x)\right]&=\left( V_{4,n}(x)\mathcal{I}+W_{4,n}(x)\frac{d}{dx}\right)S_{n-1}(x)  = S_{n}(x).  \label{laddereq2rew}
\end{align}%

\section{Differential equation}\label{Sec-DiffEqn}

In this section, we state several consequences of the equations \eqref{laddereq1rew} and \eqref{laddereq2rew}, which generalize known results for classical Laguerre polynomials to the Laguerre-Sobolev case. First, we are going to obtain a second-order differential
 equation with polynomials coefficients for $S_{n}$. The procedure  is well-known, and consists in applying the
raising operator $\mathcal{L}_{n}^{\uparrow}$\ to both sides of formula $%
\mathcal{L}_{n}^{\downarrow}\left[S_{n}\right]=S_{n-1}$. Thus, we have%
\begin{align*}
  0 = &  \mathcal{L}_{n}^{\uparrow} \left[\mathcal{L}_{n}^{\downarrow}\left[S_{n}(x)\right]\right]-S_{n}(x) \\
  = &  G_{4,n}(x)W_{4,n}(x)S_{n}^{\prime \prime }(x) \\
& + \left(F_{4,n}(x)W_{4,n}(x)+G_{4,n}(x)V_{4,n}(x)+W_{4,n}(x)G_{4,n}^\prime(x)\right)S_{n}^{\prime}(x)    \\
&+ \left(F_{4,n}(x)V_{4,n}(x)+W_{4,n}(x)F_{4,n}^\prime(x)-1\right) S_{n}(x) \\
=&\frac{q_{0,n}^2(x)}{q_{1,n}(x)q_{4,n}(x)} \; S_{n}^{\prime \prime }(x)\\
&  + \frac{q_{0,n}(x) \left(q_{1,n}(x)q_{2,n}(x)+q_{1,n}(x)q_{3,n}(x)+q_{0,n}^{\prime}(x) q_{1,n}(x)-q_{0,n}(x) q_{1,n}^{\prime}(x)\right)}{q_{4,n}(x) q_{1,n}^2(x)} \;S_{n}^{\prime}(x) \\
& + \left(\frac{q_{1,n}(x)q_{2,n}(x) q_{3,n}(x)+q_{0,n}(x) \left(q_{2,n}^{\prime}(x) q_{1,n}(x)-q_{2,n}(x) q_{1,n}^{\prime}(x)\right)}{q_{4,n}(x) q_{1,n}^2(x)}-1\right) S_{n}(x);
\end{align*} from where we conclude the following result.

\begin{theorem}\label{Poly-HolEq} The $n$th monic orthogonal polynomial with
respect to the inner product \eqref{GeneralSIP} is a polynomial
solution of the second-order linear differential equation, with polynomial  coefficients%
\begin{equation}\label{DifEq-PolyCoef}
\mathcal{P}_{2,n}(x)S_{n}^{\prime \prime }(x)+\mathcal{P}_{1,n}(x)S_{n}^{\prime }(x)+\mathcal{P}_{0,n}(x) S_{n} =0,
\end{equation}%
where%
\begin{equation}\label{DiffEqn-CoefSobolev}
\begin{aligned}
 \mathcal{P}_{2,n}(x) = &  q_{1,n}(x) q_{0,n}^2(x), \\
 \mathcal{P}_{1,n}(x)= &   q_{0,n}(x) \left(q_{1,n}(x)q_{2,n}(x)+q_{1,n}(x)q_{3,n}(x)+q_{0,n}^{\prime}(x) q_{1,n}(x)-q_{0,n}(x) q_{1,n}^{\prime}(x)\right),\\
\mathcal{P}_{0,n}(x)= & q_{1,n}(x)q_{2,n}(x) q_{3,n}(x)+q_{0,n}(x) \left(q_{2,n}^{\prime}(x) q_{1,n}(x)-q_{2,n}(x) q_{1,n}^{\prime}(x)\right)\\ &  - q_{4,n}(x) q_{1,n}^2(x),\\
\end{aligned}
\end{equation}
and
\begin{align*}
\Lcoef{\mathcal{P}_{2,n}}&= \left(6d+2,(\gamma_n+\sigma_n)\left(1+\frac{\sigma_{n-1}}{\gamma_{n-1}}\right)^2\right).\\
\Lcoef{ \mathcal{P}_{1,n}}&= \left(6d+2,-(\gamma_n+\sigma_n)\left(1+\frac{\sigma_{n-1}}{\gamma_{n-1}}\right)^2\right).\\
\Lcoef{ \mathcal{P}_{0,n}}&= \left(6d+1,n(\gamma_n+\sigma_n)\left(1+\frac{\sigma_{n-1}}{\gamma_{n-1}}\right)^2\right).\\
\end{align*}
\end{theorem}

\begin{remark}[The classical Laguere differential equation]
Note that here, $F_{1,n}(x)\equiv 1$, $G_{1,n}(x)=0$, and $\rho(x) \equiv 1$. For the rest of the expressions involved in the coefficients of the differential equation \eqref{DifEq-PolyCoef},  we have
\begin{equation*}
\begin{aligned}
\rho(x)   & \equiv 1, \; F_{1,n}(x)  \equiv  F_{2,n}(x) \equiv W_{2,n} = 1,\; G_{1,n}(x)  \equiv  G_{2,n}(x) \equiv V_{2,n} \equiv 0, \\
\Delta_n(x) & \equiv 1, \; F_{3,n}(x) = n, G_{3,n}(x)    = n(n+\alpha),\; V_{3,n}(x)  = -1 \;\text{ and} \\  W_{3,n}(x)   & = x-n-\alpha.
\end{aligned}
\end{equation*}
Thus,
\begin{equation}\label{DiffEqn-CoefJacob}
\begin{aligned}
q_{0,n}(x)&= x,\; q_{1,n}(x)=n(n+\alpha),\quad  q_{2,n}(x) = -n,\\
 q_{3,n}(x)  & = n+\alpha-x \text{ and } \\
 q_{4,n}(x) & = -1.
\end{aligned}
\end{equation}
Substituting \eqref{DiffEqn-CoefJacob} in \eqref{DiffEqn-CoefSobolev}, the reader can verify that the differential equation \eqref{DifEq-PolyCoef} becomes \eqref{Lag-DiffEqn}, this is
\begin{align*}
\mathcal{P}_{2,n}(x) &=  \gamma_nx^2, \;   \mathcal{P}_{1,n}(x) =   \gamma_nx(\alpha+1-x) \; \text{ and } \;  \mathcal{P}_{0,n}(x) =  n\gamma_nx.
\end{align*}
\end{remark}

Secondly, we can obtain the polynomial $n$th degree of the sequence $\left\{S_{n}\right\}_{n\geqslant  0}$ as the repeated action ($n$ times) of the raising differential operator on the first Sobolev-type  polynomial of the sequence (i.e., the polynomial of degree zero).



\begin{theorem}
The $n$th  Laguerre-Sobolev polynomial of   $\{S_{n}\}_{n\geqslant  0}$ is given by%
\begin{equation*}
S_{n}(x)=\left( \mathcal{L}_{n}^{\uparrow}\mathcal{L}_{n-1}^{\uparrow}\mathcal{L}_{n-2}^{\uparrow}\cdots \mathcal{L}_{1}^{\uparrow}\right) S_{0}(x), \quad \text{where $S_{0}(x)=1$.}
\end{equation*}%
\end{theorem}



\begin{proof}
Using \eqref{laddereq2rew}, the theorem follows  for $n=1$. Next, the expression for $S_{n}$ is a straightforward consequence of the definition
of the raising operator. \end{proof}


To conclude this section, we prove an interesting three-term recurrence relation, with polynomial coefficients for the Laguerre-Sobolev monic polynomials.

\begin{theorem} Under the assumptions of Theorem \ref{Poly-HolEq}, we have the recurrence relation

\begin{equation}\label{3TRR-Sob}
\begin{aligned}
 {q_{4,n+1}(x)q_{0,n}(x)} S_{n+1}(x) =&  \left[ q_{3,n+1}(x)q_{0,n}(x) - q_{2,n}(x)q_{0,n+1}(x) \right]S_{n}(x)\\ & + {q_{1,n}(x)q_{0,n+1}(x)} S_{n-1}(x),
\end{aligned}%
\end{equation}
where the explicit formulas of the coefficients are given in Theorem  \ref{lemma44}.
\end{theorem}

\begin{proof}
From  \eqref{laddereq1} and   \eqref{laddereq2} for $n+1$,  we have

\begin{align*}
F_{4,n}(x)S_{n}(x)+G_{4,n}(x) S'_{n}(x) &=S_{n-1}(x), \\
V_{4,n+1}(x) S_{n}(x)+W_{4,n+1}(x)S'_{n}(x) &= S_{n+1}(x).
\end{align*}
Substituting the coefficients by its explicit expressions given in Theorem  \ref{lemma44}, we obtain
\begin{align*}
q_{2,n}(x)S_{n}(x)+q_{0,n}(x)S_{n}^{\prime}(x) & = q_{1,n}(x)S_{n-1}(x); \\
q_{3,n+1}(x)S_{n}(x)+q_{0,n+1}(x)S_{n}^{\prime}(x) & = q_{4,n+1}(x)S_{n}(x).
\end{align*}
Multiplying the equations by $q_{0,n+1}(x)$ and $q_{0,n}(x)$ respectively and subtracting the first equation from the second one to eliminate the derivative term, we get
\begin{align*}
 \left(q_{3,n+1}(x)q_{0,n}(x)-q_{2,n}(x)q_{0,n+1}(x)\right)&S_n(x) = q_{4,n+1}(x)q_{0,n}(x)S_{n+1}(x)-q_{1,n}(x)q_{0,n+1}(x)S_{n-1}(x),
\end{align*}
which is the required formula.
\end{proof}

\begin{remark}[The classical Laguerre three-term recurrence relation] Under the assumptions of  Remark \ref{Remark-Operators},   substituting \eqref{DiffEqn-CoefJacob} in \eqref{3TRR-Sob}, the reader can verify that the three-term recurrence relation \eqref{3TRR-Sob} becomes \eqref{Lag3TRR}, this is
\begin{align*}
\frac{q_{3,n+1}(x)q_{0,n}(x) - q_{2,n}(x)q_{0,n+1}(x)}{q_{4,n+1}(x)q_{0,n}(x)} & = x-2n-\alpha-1 \quad \text{and}\\
\frac{q_{1,n}(x)q_{0,n+1}(x)}{q_{4,n+1}(x)q_{0,n}(x)}  & = -n(n+\alpha).
\end{align*}
\end{remark}

\section{Electrostatic model}

 Let $\rho(x)$ be as in \eqref{rho-def} and $d\mu_{\rho}(x)= \rho(x)d\mu^{\alpha}(x)$. Note that $\rho$ is a polynomial of degree $d=\sum_{j=1}^N(d_j+1)$ which is positive on $(0,+\infty)$.  If $n >d$, from  \eqref{Sobolev-Orth}, {$\{S_n\}_{n\geqslant 0}$}  satisfies the following quasi-orthogonality relations with respect to $\mu_{\rho}$
\begin{equation*}
\Ip{S_n}{f}_{\mu_{\rho}}  = \Ip{S_n}{{\rho} f}_{\alpha} = \int_{0}^{+\infty} S_n(x) f(x) {\rho}(x) d\mu^\alpha(x) = \IpS{S_n}{{\rho} f}=0 ,
\end{equation*}
for $f\in \PP_{n-d-1}$, where $\PP_n$ is the linear space of polynomials with real coefficients and degree less than or equal to $n\in \ZZp$. Hence,  \emph{the polynomial $S_n$ is quasi-orthogonal of order $d$ with respect to $\mu_{\rho}$} and by this argument, we get that  $S_n$ has at least $(n-d)$ changes of sign in $(0,+\infty)$.

From the results obtained in \cite[(1.10)]{LopMarVan95}, \cite{DiPiPe20}, and \cite{DiHerPi23}, it seems  that the number of zeros located in the interior of the support of the measure is closely related to $d^*=\#(I_+)$. Recall that  $d^*$ is  the number of terms in the discrete part of $\IpS{\cdot}{\cdot}$ (i.e., $\lambda_{j,k}>0$).

Let us continue by giving the definition of sequentially-ordered  Sobolev inner product, which was first introduced in \cite[Def. 1]{DiPiPe20} and then refined in \cite[Def. 1]{DiHerPi23}.

\begin{definition}[\textbf{Sequentially-ordered Sobolev inner product}]\label{Set-SOrdered-General}
 Consider a discrete Sobolev inner product in the general form  \eqref{GeneralSIP} and assume that $d_1\leq d_2\leq \dots\leq d_N$ without loss of generality. We say that a discrete Sobolev inner product is \emph{sequentially ordered} if the conditions
	\begin{equation*}
		\Delta_{k}\cap \inter{\ch{\cup_{i=0}^{k-1}\Delta_i}} =\emptyset, \quad \quad k=1,2,\dots, d_N,
	\end{equation*}		
	hold, where
$$\Delta_k=\funD{\ch{\supp{\mu}\cup\{c_j: \lambda_{j,0}>0\}}}{k=0}{\ch{\{c_j: \lambda_{j,k}>0\}}}{1\leq k\leq d_N.}$$
\end{definition}
Note that  $\Delta_k$ is the convex hull of the support of the measure associated with the $k$-th order derivative in the Sobolev inner product   \eqref{GeneralSIP}.

Hereinafter, we will focus on sequentially ordered discrete Sobolev inner products that have only one derivative order at each mass point. This means that
\eqref{GeneralSIP} takes the form
\begin{align}\label{LaguerreSIP}
	\IpS{f}{g}=\!\!\!{\int_0^{+\infty}} f(x) g(x)x^{\alpha}e^{-x}dx+\sum_{j=1}^{N} \lambda_{j} f^{(d_j)}(c_j)g^{(d_j)}(c_j),
\end{align}	
where $\lambda_j:=\lambda_{j,d_j}>0$, and  $c_j<0$, for $j=1,2,\dots,N$.

The following {lemma shows} our reasons for this assumption.

\begin{lemma}[{\cite[Cor 2-1]{DiHerPi23}-\cite[Prop. 4]{DiPiQui22}}]\label{Th_ZerosSimp}  If \eqref{LaguerreSIP} is  a sequentially-ordered discrete Sobolev inner product, then the following statements hold:
\begin{itemize}
\item[1.] Every point $c_j$ attracts exactly one zero of $S_n$ for sufficiently large $n$, while the remaining $n-N$ zeros are contained in $(0,+\infty)$. This means:
For every $r>0$, there exists a natural value $\mathcal{N}$ such that if $n \geq \mathcal{N}$, then the $n$ zeros of $S_n$ $\left\{\xi_i\right\}_{i=1}^n$ satisfy
$\xi_j \in B\left(c_j, r\right)$ for $j=1, \ldots, N \quad$ and $\quad \xi_i \in(0, +\infty)$ for $i=N+1, N+2, \ldots, n$.
\item[2.] The zeros of $S_n$ are real and simple for large-enough values of $n$.
\end{itemize}
\end{lemma}

In the rest of this section we will assume that the zeros of $S_n$ are simple. Observe  that  sequentially-ordered  Sobolev inner products provide us with a wide class of Sobolev inner products  such that the zeros of the corresponding orthogonal polynomials are simple.  Therefore, for all $n$ sufficiently large, we have
\begin{equation*}
\begin{array}{lll}
 S_{n}^{\prime}(x)={\displaystyle\sum\limits_{i=1}^{n}} \prod\limits_{\substack{ j=1,  \\ j\neq i}}^{n}{(x-x_{n,j})}, &  &
 S_{n}^{\prime\prime}(x)={\displaystyle \sum\limits_{\substack{ i,j=1,  \\ j\neq i}}^{n}\prod\limits_{\substack{ l=1,  \\l\neq i,j}}^{n}(x-x_{n,l})}, \\
&  &  \\
S_{n}^{\prime}(x_{n,k})={\displaystyle\prod\limits_{\substack{ j=1, \\ j\neq k}}^{n}(x_{n,k}-x_{n,j})}, &  & S_{n}^{\prime\prime}(x_{n,k})={\displaystyle2\sum\limits_{\substack{ i=1, \\ i\neq k}}^{n}\prod\limits_{\substack{ l=1,  \\ l\neq i,k}}^{n}(x_{n,k}-x_{n,l})}.
\end{array}
\end{equation*}

Now, we evaluate the polynomials $\mathcal{P}_{2,n}(x),\; \mathcal{P}_{1,n}(x) $ and $\mathcal{P}_{0,n}(x) $ in \eqref{DifEq-PolyCoef} at $x_{n,k}$, where $\left\{ x_{n,k}\right\}_{k=1}^{n}$ are the zeros of $S_{n}(x)$ arranged in an increasing order. Then, for $k=1,2,\dots,n$; we get
\begin{align}\nonumber
0=& \mathcal{P}_{2,n}(x_{n,k}) S_{n}^{\prime\prime}(x_{n,k})+ \mathcal{P}_{1,n}(x_{n,k})  S_{n}^{\prime}(x_{n,k})+\mathcal{P}_{0,n}(x_{n,k})  S_{n}(x_{n,k})\\ \nonumber
 =&  \mathcal{P}_{2,n}(x_{n,k}) S_{n}^{\prime\prime}(x_{n,k})+ \mathcal{P}_{1,n}(x_{n,k})  S_{n}^{\prime}(x_{n,k}).\\ \label{CritPoint}
0= &\frac{ S_{n}^{\prime\prime}(x_{n,k})}{ S_{n}^{\prime}(x_{n,k})} +\frac{\mathcal{P}_{1,n}(x_{n,k})}{\mathcal{P}_{2,n}(x_{n,k}) }=2  \sum_{{{i=1} \atop {i \neq k}}}^{n} {\frac{1}{x_{n,k}-x_{n,i}}} +\frac{\mathcal{P}_{1,n}(x_{n,k})}{\mathcal{P}_{2,n}(x_{n,k}) }.
\end{align}

From Theorem \ref{lemma44}, Theorem \ref{Poly-HolEq}, and Lemma \ref{LemmaDeterminant}
\begin{align}\nonumber
\frac{ \mathcal{P}_{1,n}(x) }{ \mathcal{P}_{2,n}(x) } & = \frac{q_{1,n}(x)q_{2,n}(x)+q_{1,n}(x)q_{3,n}(x)+q_{0,n}^{\prime}(x) q_{1,n}(x)-q_{0,n}(x) q_{1,n}^{\prime}(x)}{q_{1,n}(x) q_{0,n}(x)}\\ \nonumber
 & =\frac{q_{2,n}(x)+q_{3,n}(x)}{q_{0,n}(x)}+  \frac{ q_{0,n}^{\prime}(x)}{q_{0,n}(x)}- \frac{ q_{1,n}^{\prime}(x)}{q_{1,n}(x)}\\ \nonumber
  & = 2\frac{\rho^{\prime}(x)}{\rho(x)}+\frac{\Delta_{2,n}(x)+\Delta_{3,n}(x)}{ x \rho(x)\delta_n(x)} + \frac{\Delta_{n}^{\prime}(x)}{\Delta_{n}(x)}+\frac{1}{x}- \frac{\Delta_{1,n}^{\prime}(x)}{\Delta_{1,n}(x)}\\
    & = 3\frac{\rho^{\prime}(x)}{\rho(x)}+\frac{\varphi_{2,n}(x)+\varphi_{3,n}(x)}{ x\rho_N(x)\delta_n(x)}+\frac{ \delta_{n}^{\prime}(x)}{\delta_{n}(x)} + \frac{1}{x} - \frac{ \varphi_{1,n}^{\prime}(x)}{\varphi_{1,n}(x)}-\frac{ \rho_{d-N}^{\prime}(x)}{\rho_{d-N}(x)}. \label{DescompCociente}
\end{align}
Let us write,  $\dsty  \frac{\rho^{\prime}(x)}{\rho(x)}  = \sum_{j=1}^{N} \frac{d_j+1}{x-c_j},\qquad \frac{\rho_{d-N}^{\prime}(x)}{\rho_{d-N}(x)}  = \sum_{j=1}^{N} \frac{d_j}{x-c_j}.$

From \eqref{deglemaDetrminant} and Remark \ref{remark31},  $\dsty \varphi_{2,n}(x)+\varphi_{3,n}(x)$ and  $\dsty x\rho_N(x)\delta_n(x)$ are polynomials of degree $d+N+1$ and  leading coefficient $\mp(1+\sigma_{n-1}/\gamma_{n-1})$ respectively. Then, their quotient can be rewritten as

$$
\frac{\varphi_{2,n}(x)+\varphi_{3,n}(x)}{x\rho_N(x)\delta_n(x)} = -1+\frac{\psi_1(x)}{ \psi_2(x)};
$$
where $\psi_2(x)=x\rho_N(x)\delta_n(x)$ and $\psi_1$ is a polynomial of degree at most $d+N$.

Based on the results of our numerical experiments, in the remainder of the section, we will assume certain restrictions with respect to some functions and parameters involved in \eqref{DescompCociente}. In that sense, we suppose that
\begin{enumerate}
 \item The zeros of $\delta_{n}$ are real simple and different from the zeros of $S_n$, the mass points and zero, i.e., $u_i\in \mathbb{R}\setminus \left(\{x_{n,k}\}_{k=1}^n\cup \{c_j\}_{j=1}^N\cup\{0\}\right)$ for $i=1,2,\dots,d$ and $u_i\neq u_j$ unless $i=j$. Therefore,
  \begin{align*}
      \delta_{n}(x)= \left(1+\frac{\sigma_{n-1}}{\gamma_{n-1}}\right) \prod_{i=1}^{d} (x-u_i),\\
      \frac{\delta_{n}^{\prime}(x)}{\delta_{n}(x)}   = \sum_{i=1}^{d} \frac{1}{x-u_i}.\\
  \end{align*}
    Thus,
    $$
        \frac{\psi_1(x)}{ \psi_2(x)} = \frac{r(0)}{x}+\sum_{j=1}^{N} \frac{r(c_j)}{x-c_j}+\sum_{i=1}^{d} \frac{r(u_i)}{x-u_i}, \quad \text{where } r(x)=\frac{\psi_1(x)}{ \psi_2^{\prime}(x)}.
    $$

  \item Let $\dsty \varphi_{1,n}(x)=(\gamma_n+\sigma_n) \prod_{j=1}^{N_1} (x-e_{j})^{\ell_{4,j}}$, where  $e_{j} \in \CC \setminus {\ch{[0,+\infty)\cup\{c_1,\dots,c_{N}\}}}$ for all $j=1,\dots, N-1$; and  $\dsty \sum_{j=1}^{N_1}\ell_{4,j}=d+N$. Therefore, $\dsty  \frac{\varphi_{1,n}^{\prime}(x)}{\varphi_{1,n}(x)}   = \sum_{j=1}^{N_1} \frac{\ell_{4,j}}{x-e_{j}}.$

 \item The function $r$ satisfies $r(0),r(u_i)>-1$ for $i=1,2,\dots d$ and $r(c_j)>-2d_j-3$ for $j=1,2,\dots,N$. Substituting the previous decompositions into \eqref{DescompCociente}, we have
$$
  \frac{ \mathcal{P}_{1,n}(x) }{ \mathcal{P}_{2,n}(x) }=-1+
\frac{\ell_{1}}{x}+\sum_{j=1}^{N} \frac{\ell_{2,j}}{x-c_j}+\sum_{j=1}^{d} \frac{\ell_{3,j}}{x-u_j}-  \sum_{j=1}^{N_1} \frac{\ell_{4,j}}{x-e_{j}},
$$where,  $\ell_{1}=1+r(0)$, $\ell_{2,j}=2d_j+r(c_j)+3$ and   $\ell_{3,j}= r(u_j)+1$ are positive values.

\end{enumerate}

From  \eqref{CritPoint}, for $k=1,\dots, n$, we get
\begin{align}\nonumber
 0& =   \sum_{{{i=1} \atop {i \neq k}}}^{n} {\frac{1}{x_{n,k}-x_{n,i}}} -\frac{1}{2} + \frac{\ell_{1}}{2x_{n,k}} \\
 & \phantom{= } + \frac{1}{2}\sum_{j=1}^{N} \frac{\ell_{2,j}}{x_{n,k}-c_j}+\frac{1}{2}\sum_{j=1}^{d} \frac{\ell_{3,j}}{x_{n,k}-u_j}+\frac{1}{2}\sum_{j=1}^{N_1} \frac{\ell_{4,j}}{e_{j}-x_{n,k}}.  \label{PtosCrit-Potencial}
\end{align}

Let $\dsty \overline{\omega}=  (\omega_1,\omega_2,\cdots,\omega_n)$, $\overline{x}_n=  (x_{n,1},x_{n,2},\cdots,x_{n,n})$ and denote
\begin{align}  \label{LogPotential}
\mathbf{E}(\overline{\omega}) := & \;\sum_{1\leq k<j \leq n }\log{\frac{1}{\omega_j-\omega_k}} +\mathbf{F}(\overline{\omega})+ \mathbf{G}(\overline{\omega}), \;  \\  \nonumber 
\mathbf{F}(\overline{\omega}):=&\frac{1}{2}\sum_{k=1}^{n}\left(|\omega_k| + \log{\frac{1}{|\omega_k|^{\ell_1}}} +\sum_{j=1}^{N}  \log{\frac{1}{|c_j-\omega_k|^{\ell_{2,j}}}}\right), \\
\mathbf{G}(\overline{\omega}):=&\frac{1}{2}\sum_{k=1}^{n}\left(   \sum_{j=1}^{d}   \log{\frac{1}{|u_j-\omega_k|^{\ell_{3,j}}}}{ + \sum_{j=1}^{N_1}  \log{\frac{1}{|\omega_k-e_j|^{\ell_{4,j}}}}}\right).\nonumber 
\end{align}

Let us  introduce the following electrostatic interpretation:
\begin{quotation}
\emph{Consider the system of $n$ movable positive unit charges at $n$ distinct points of the real line, $\{ \omega_1,\omega_2,\cdots,\omega_n\}$, where their interaction follows the logarithmic potential law (that is, the force is inversely proportional to the relative distance)  in the presence of the total external potential $\mathbf{H}_n(\overline{\omega})=\mathbf{F}(\overline{\omega})+ \mathbf{G}(\overline{\omega})$. Then, $\mathbf{E}(\overline{\omega})$ is the total energy  of this system. }
\end{quotation}

Following the notations introduced in \cite[Sec. 2]{Ism00-A}, the Laguerre-Sobolev inner product creates two external fields. One is a long-range field whose potential is $\mathbf{F}(\overline{\omega})$. The other  is a short-range field whose potential is  $\mathbf{G}(\overline{\omega})$. So, the total external potential $\mathbf{H}_n(\overline{\omega})$  is the sum of the short and long-range potentials, which is dependent on $n$ (varying external potential).

Therefore, for each $k=1,\dots, n$; we have $\dsty \frac{\partial  \mathbf{E}}{\partial \omega_k}(\overline{x}_n)=   0$, this is, the zeros of $S_n$ are the zeros  of  the gradient  of the total potential of energy $\mathbf{E}(\overline{\omega})$ ($\nabla \mathbf{E}(\overline{x}_n)=0$).

\begin{theorem}\label{PositiveHessian}
The zeros of $S_n(x)$ are a local minimum of $\mathbf{E}(\overline{\omega})$,  if  for all $k=1,\dots, n;$
\begin{enumerate}
  \item $\dsty \frac{\partial  \mathbf{E}}{\partial \omega_k}(\overline{x}_n)=0$.
  \item $\dsty \frac{\partial^2  \mathbf{H}_n}{\partial {\omega}_k^2}(\overline{x}_n) =\frac{\partial^2  \mathbf{F}}{\partial {\omega}_k^2}(\overline{x}_n)+\frac{\partial^2 \mathbf{G}}{\partial {\omega}_k^2}(\overline{x}_n) {=\ - \frac{1}{2} \left( \frac{ \mathcal{P}_{1,n}(x) }{ \mathcal{P}_{2,n}(x) }\right)'> 0}. $
\end{enumerate}
\end{theorem}
\begin{proof}

The Hessian matrix of $\mathbf{E}$ at $\overline{x}_n$ is given by
\begin{equation}\label{Hessian_MSob}
\nabla_{\overline{\omega}\,\overline{\omega}}^2 \mathbf{E}(\overline{x}_n) = \begin{cases}
\dsty \frac{\partial^2 \mathbf{E}}{\partial {\omega}_k\partial {\omega}_j}(\overline{x}_n)=-\frac{1}{(x_{n,k}-x_{n,j})^{2}}, & \text{if } \;k\neq j,\\
\dsty \frac{\partial^2 \mathbf{E}}{\partial {\omega}_k^2}(\overline{x}_n)= \sum_{{{i=1} \atop {i \neq k}}}^{n} \frac{1}{(x_{n,k}-x_{n,i})^2}+\frac{\partial^2 \mathbf{H}_n }{\partial {\omega}_k^2}(\overline{x}_n), & \text{if } \;k = j.
\end{cases}
\end{equation}
{ Since \eqref{Hessian_MSob} is a symmetric real matrix, its eigenvalues are real. Therefore, using Gershgorin's Theorem \cite[Th. 6.1.1]{Horn90}, the eigenvalues $\lambda$ of the Hessian at $\overline{x}$ satisfy
\begin{align*}
    \left|\lambda-\sum_{{{i=1} \atop {i \neq k}}}^{n} \frac{1}{(x_{n,k}-x_{n,i})^2}-\frac{\partial^2 \mathbf{H}_n}{\partial {\omega}_k^2}(\overline{x}_n)\right|\leq  \sum_{{{j=1} \atop {j \neq k}}}^{n}\left|\frac{\partial^2 \mathbf{E}}{\partial {\omega}_k\partial {\omega}_j}(\overline{x}_n)\right|= \sum_{{{i=1} \atop {i \neq k}}}^{n} \frac{1}{(x_{n,k}-x_{n,i})^2}.
\end{align*}
for some  $k=1,2,\dots,n$. Then, we have
\begin{align*}
    \lambda\geq \frac{\partial^2 \mathbf{H}_n}{\partial {\omega}_k^2}(\overline{x}_n)>0.
\end{align*}
}

\end{proof}

The computations of the following examples have been performed by using the symbolic computer algebra system \emph{Maxima} \cite{OchMak20}.

In all cases, we fixed $n=12$ and consider sequentially-ordered Sobolev inner product  (see Definition \ref{Set-SOrdered-General} and  Lemma \ref{Th_ZerosSimp}). From \eqref{PtosCrit-Potencial}, it is obvious that $\nabla \mathbf{E}(\overline{x}_{12})=0$, where $\overline{x}_{12}=  (x_{12,1},x_{12,2},\cdots,x_{12,n})$ and $S_{12}(x_{12,k})=0$ for $k=1,\;2,\dots,\,12$. Under the above condition, $\overline{x}_{12}$ is a local minimum (maximum) of $\mathbf{E}$ if the corresponding Hessian matrix at $\overline{x}_{12}$ is positive (negative) definite, in any other case $\overline{x}_{12}$ is said to be a saddle point. We recall that a square matrix is positive (negative) definite if all its eigenvalues are positive (negative).

\begin{example}[Case in which the conditions of the Theorem \ref{PositiveHessian} are satisfied]\

\begin{enumerate}
  \item Laguerre-Sobolev inner product $\dsty \IpS{f}{g}=\int_{0}^{+\infty} f(x)g(x)x^{11}e^{-x}dx+f^{\prime}(-2)g^{\prime}(-2)$.
  \item Zeros of $S_{12}(x)$.
   \begin{align*}
     \overline{x}_{12}=& \left(3.0537,\, 5.16053,\, 7.53124,\, 10.2434,\, 13.3451,\, 16.8869,\,\right.  \\
       &\left.    20.9337,\, 25.5751,\, 30.9455,\, 37.2657,\, 44.9569,\, 55.0972 \right).
   \end{align*}
  \item Total potential of energy  $\dsty \mathbf{E}(\overline{\omega}) =\sum_{1\leq k<j \leq 12 }\log{\frac{1}{|\omega_j-\omega_k|}} + \mathbf{F}(\overline{\omega})+\mathbf{G}(\overline{\omega})$,   where
  \begin{align*}
\mathbf{F}(\overline{\omega})= &\frac{1}{2}\sum_{k=1}^{12} \left(|\omega_k| +12\log{\frac{1}{\left|\omega_k \right|}}  + \log{\frac{1}{\left| \omega_k+2 \right|^3} }\right),\; \\
 \mathbf{G}(\overline{\omega})= & \frac{1}{2}\sum_{k=1}^{12} \log\left| (\omega_k + 0.528573)(\omega_k + 1.7501)(\omega_k - 1.40334)\right|.
\end{align*}

  \item From \eqref{PtosCrit-Potencial}, $\dsty \frac{\partial  \mathbf{E}}{\partial \omega_j}(\overline{x}_{12})=0, $ for $j=1,\dots, 12$.

  \item Computing the corresponding Hessian matrix at $\overline{x}_{12}$, we have that the {approximate eigenvalues}  are
\begin{align*}
&   \left\{
0.0127\;
0.0304\;
0.0517\;
0.0778\;
0.1102\;
0.1509\;\;\right. \\
&   \left.
0.2033\;
0.2722\;
0.3653\;
0.495\;
0.6825\;
0.9661 \right\}.
\end{align*}
\end{enumerate}

Thus, Theorem \ref{PositiveHessian} holds for this example, and we have the required local electrostatic equilibrium distribution.
\end{example}

\begin{example}[Case in which the conditions of the Theorem \ref{PositiveHessian} are satisfied]\

\begin{enumerate}
  \item Laguerre-Sobolev inner product $\dsty \IpS{f}{g}=\int_{0}^{+\infty} f(x)g(x)x^{14}e^{-x}dx+f^{\prime}(-2)g^{\prime}(-2)$.
  \item Zeros of $S_{12}(x)$.
   \begin{align*}
     \overline{x}_{12}= & \left(4.7832,\, 7.23584,\, 9.92786,\, 12.9448,\, 16.3404,\, 20.1693, \right.  \\
       &\left.    24.4992,\, 29.4232,\, 35.0794,\, 41.6941,\, 49.6983,\, 60.1956   \right).
   \end{align*}

  \item Total potential of energy  $\dsty \mathbf{E}(\overline{\omega}) =\sum_{1\leq k<j \leq 12 }\log{\frac{1}{|\omega_j-\omega_k|}} + \mathbf{F}(\overline{\omega})+\mathbf{G}(\overline{\omega})$,   where
  \begin{align*}
\mathbf{F}(\overline{\omega})= &\frac{1}{2}\sum_{k=1}^{12} \left(|\omega_k| +15\log{\frac{1}{\left|\omega_k \right|}}  + \log{\frac{1}{\left| \omega_k+2 \right|^3} }\right),\; \\
 \mathbf{G}(\overline{\omega})= & \frac{1}{2}\sum_{k=1}^{12} {\log\left| (\omega_k +1.87468) \tau(\omega_k)\right|} \;
\text{ and }  \; \tau(x)  =  4.25297+4.10532 x+{{x}^{2}} >  0.
\end{align*}

  \item From \eqref{PtosCrit-Potencial}, $\dsty \frac{\partial  \mathbf{E}}{\partial \omega_j}(\overline{x}_{12})=0, $ for $j=1,\dots, 12$.
  \item Computing the corresponding Hessian matrix at $\overline{x}_{12}$, we have that the {approximate eigenvalues}  are
\begin{align*}
&   \left\{
0.0152, \;
0.0344, \;
0.0576, \;
0.0861, \;
0.1219, \;
0.1678, \;\right. \\
&   \left.
0.2279, \;
0.3094, \;
0.4241, \;
0.5942, \;
0.8665, \;
1.3566 \right\}.
\end{align*}
\end{enumerate}

Thus, Theorem \ref{PositiveHessian} holds for this example, and we have the required local electrostatic equilibrium distribution.
\end{example}
\begin{example}[Case in which the conditions of the Theorem \ref{PositiveHessian} are satisfied ]\

\begin{enumerate}
  \item Laguerre-Sobolev inner product $\dsty \IpS{f}{g}=\int_{0}^{+\infty} f(x)g(x)x^{11}e^{-x}dx+f^{\prime\prime}(-2)g^{\prime\prime}(-2)$.
  \item Zeros of $S_{12}(x)$.
   \begin{align*}
     \overline{x}_{12}=&   \left(3.35093,\,
   5.41033,\,
   7.75809,\,
   10.456,\,
   13.5478,\,
   17.0825,\, \right.\\
   & \left. 21.1239,\,
   25.7612,\,
   31.1283,\,
   37.4459,\,
   45.1347,\,
   55.2729\right).
   \end{align*}

  \item Total potential of energy  $\dsty \mathbf{E}(\overline{\omega}) =\sum_{1\leq k<j \leq 12 }\log{\frac{1}{|\omega_j-\omega_k|}} + \mathbf{F}(\overline{\omega})+\mathbf{G}(\overline{\omega})$,   where
  \begin{align*}
\mathbf{F}(\overline{\omega})= &\frac{1}{2}\sum_{k=1}^{12} \left(|\omega_k| +12\log{\frac{1}{\left|\omega_k \right|}}  + \log{\frac{1}{\left| \omega_k+2 \right|^4} }\right),\; \\
 \mathbf{G}(\overline{\omega})= & \frac{1}{2}\sum_{k=1}^{12} \log\left| (\omega_k +0.0989292)(\omega_k + 1.64715)\tau(\omega_k)\right| \\
& \text{ and }  \tau(x)  =  5.72898-1.63056 x+{{x}^{2}} >  0.
\end{align*}

  \item From \eqref{PtosCrit-Potencial}, $\dsty \frac{\partial  \mathbf{E}}{\partial \omega_j}(\overline{x}_{12})=0, $ for $j=1,\dots, 12$.

  \item Computing the corresponding Hessian matrix at $\overline{x}_{12}$, we have that the {approximate eigenvalues}  are
\begin{align*}
&   \left\{
0.0126\;
0.0303\;
0.0516\;
0.0777\;
0.1101\;
0.151\;\;\right. \\
&   \left.
0.2038\;
0.2737\;
0.3689\;
0.5042\;
0.7066\;
1.0321\right\}.
\end{align*}
\end{enumerate}

Thus, Theorem \ref{PositiveHessian} holds for this example, and we have the required local electrostatic equilibrium distribution.

\end{example}
\begin{example}[Case in which the conditions of the Theorem \ref{PositiveHessian} are satisfied]\

\begin{enumerate}
  \item Laguerre-Sobolev inner product $$\dsty \IpS{f}{g}=\int_{0}^{+\infty} f(x)g(x)x^{14}{e^{-x}}dx+f^{\prime}(-1)g^{\prime}(-1) + f^{\prime\prime}(-2)g^{\prime\prime}(-2).$$
  \item Zeros of $S_{12}(x)$.
   \begin{align*}
     \overline{x}_{12}=& \left(   4.78339,\,7.23607,\,9.9281,\,12.9451,\,16.3407,\,20.1695, \right.  \\
       &\left.  24.4995,\,29.4235,\,35.0797,\,41.6944,\,49.6986,\,60.196 \right).
   \end{align*}

  \item Total potential of energy  $\dsty \mathbf{E}(\overline{\omega}) =\sum_{1\leq k<j \leq 12 }\log{\frac{1}{|\omega_j-\omega_k|}} + \mathbf{F}(\overline{\omega})+\mathbf{G}(\overline{\omega})$,   where
  \begin{align*}
\mathbf{F}(\overline{\omega})= &\frac{1}{2}\sum_{k=1}^{12} \left( \log{\frac{1}{\left| \omega_k\right|^{15}} }+ \log{\frac{1}{\left|\omega_k+1 \right|^3} }+\log{\frac{1}{\left|\omega_k+2 \right|^4} }\right),\; \\
 \mathbf{G}(\overline{\omega})= & \frac{1}{2}\sum_{k=1}^{12} \log\left|{(\omega_k +0.933652)\tau_1(\omega_k)\tau_2(\omega_k)\tau_3(\omega_k) }\right|\\
& \text{and } \quad  \begin{aligned}
\tau_1(x) & =   1.07168+2.06058 x+{{x}^{2}}>0,\\
\tau_2(x) & =   3.19751+3.56774 x+{{x}^{2}}>0,\\
\tau_3(x) & =   4.99621+4.42482 x+{{x}^{2}}>0.
\end{aligned}
\end{align*}
  \item From \eqref{CritPoint}, $\dsty \frac{\partial  \mathbf{E}}{\partial \omega_j}(\overline{x}_{12})=0, $ for $j=1,\dots, 12$.
  \item Computing the corresponding Hessian matrix at $\overline{x}_{12}$, we have that the {approximate eigenvalues}  are
\begin{align*}
&   \left\{0.0117,\;
0.0278,\;
0.0469,\;
0.0699,\;
0.0978,\;
0.1322,\right. \\
&   \left. 0.1752,\;
0.2301,\;
0.3016,\;
0.3973,\;
0.5292,\;
0.7179 \right\}.
\end{align*}
   \end{enumerate}

Thus, Theorem \ref{PositiveHessian} holds for this example, and we have the required local electrostatic equilibrium distribution.

\end{example}


\begin{example}[Case in which the conditions of the Theorem \ref{PositiveHessian} are not satisfied]\label{Examp_NoS}\

\begin{enumerate}
  \item Laguerre-Sobolev inner product $$\dsty \IpS{f}{g}=\int_{0}^{+\infty} f(x)g(x)e^{-x}dx+f^{\prime}(-1)g^{\prime}(-1)+f^{\prime\prime}(-2)g^{\prime\prime}(-2).$$
  \item Zeros of $S_{12}(x)$.
   \begin{align*}
     \overline{x}_{12}=& \left(-2.86242,\, -1.69526,\, 0.284629,\, 1.36447,\, 3.03668,\, 5.23686,\, \right. \\
       & \left. 7.98826,\, 11.3572,\, 15.4574,\, 20.4841,\, 26.8154,\, 35.422 \right).
         \end{align*}
  \item Total potential of energy  $\dsty \mathbf{E}(\overline{\omega}) =\sum_{1\leq k<j \leq 12 }\log{\frac{1}{|\omega_j-\omega_k|}} + \mathbf{F}(\overline{\omega})+\mathbf{G}(\overline{\omega})$,   where
      \begin{align*}
\mathbf{F}(\overline{\omega})= &\frac{1}{2}\sum_{k=1}^{12} \left( \log{\frac{1}{\left| \omega_k-1 \right|} }+ \log{\frac{1}{\left| \omega_k+1 \right|} }+ \log{\frac{1}{\left| \omega_k-2 \right|^3} }\right),\\
 \mathbf{G}(\overline{\omega})= & \frac{1}{2}\sum_{k=1}^{12} \log\left|{(\omega_k +1.44591)(\omega_k +1.78194)(\omega_k+2.7333) \tau_1(\omega_k)\tau_2(\omega_k) }\right|\;\\
& \text{and } \quad 
\begin{aligned}
\tau_1(x)  & =   0.440466+1.23097 x+{{x}^{2}} >  0,\\
\tau_2(x)  & =   2.76889+3.19398 x+{{x}^{2}} >  0.
\end{aligned}
\end{align*}
  \item From \eqref{PtosCrit-Potencial}, $\dsty \frac{\partial  \mathbf{E}}{\partial \omega_j}(\overline{x}_{12})=0, $ for $j=1,\dots, 12$.
  \item Computing the corresponding Hessian matrix at $\overline{x}_{12}$, we have that the {approximate eigenvalues}  are
\begin{align*}
&   \left\{-45.8083, \,
-27.1075 , \,
0.0188, \,
0.0473, \,
0.0853, \,
0.1377, \, \,\right. \\
&   \left. 0.213, \,
0.3272, \,
0.5154, \,
0.8688, \,
1.7428, \,
7.4559\right\}.
\end{align*}
\end{enumerate}
Then,  $\overline{x}_{12}$ is a saddle point  of $\mathbf{E}(\overline{\omega})$.
\end{example}


\bigskip

\begin{remark}
As can be noticed,  in some cases the configuration given by the external field includes complex points, they correspond to $e_j$. Specifically, in the examples, these points are given as the zeros of $\tau(x)$. Since $\phi_{1,n}(x)$ is a polynomial of real coefficients, its non-real zeros arise as complex conjugate pairs. Note that
$$
\frac{a}{x-z}+\frac{a}{x-\overline{z}} = a\frac{2x+2\Re{z}}{x^2+2\Re{z}+|z|^2};
$$
where $\Re z$ denotes the real part of $z$. The anti derivative of the previous expression is $a\ln(x^2+2\Re z+ |z|^2)$. This means in our current case that the presence of complex roots does not change  the formulation of the energy function.
\end{remark}

\begin{remark}
Theorem \ref{PositiveHessian} {gives} a general condition to determine whether the electrostatic model  is an   extension of the classical cases. However, in  Example \ref{Examp_NoS}, the Hessian has {two negatives eigenvalues corresponding to the first two variables $\omega_1$ and $\omega_2$ }.

Therefore, we do not have the nice interpretation given in Theorem \ref{PositiveHessian}. However, note that the rest of the eigenvalues are positive, {which means that the number
$${\frac{\partial^2 \mathbf{H}_n}{\partial {\omega}_k^2}}(\overline{x}_n)$$
remains positive for $k=3,\ldots, 12$.} In this case, the potential function exhibits a saddle point. The presence of the saddle point is somehow justified by the attractor points $-1.44591$, $-1.78194$ and $-2.7333$ having two zeros in its vicinity $x_{12,1} \approx -2.86242$ and $x_{12,2} \approx -1.69526$. In this case, we are able to give an interpretation of the position of the zeros by considering a problem of conditional extremes.

Assume that when checking the Hessian we obtained that the eigenvalues $\lambda_{n,i}$, for the indexes $i \in \mathcal{E}\subset \lbrace 1,2,\ldots,n \rbrace$, are negative or zero. Without loss of generality, assume that this happens for the first $m_{\mathcal{E}}=|\mathcal{E}|$ variables.  This is a saddle point. However,  the   rest of the eigenvalues are positive, which means that the truncated Hessian $\nabla_{\omega_{m_{\mathcal{E}}}\omega_{m_{\mathcal{E}}}}^2 \mathbf{E}$ formed by taking the last $n-m_{\mathcal{E}}$ rows and columns of $\nabla_{\overline{\omega}\,\overline\omega}^2 \mathbf{E}_R$ is a positive definite matrix by the same arguments used in the proof of Theorem \ref{PositiveHessian}.

Thus, let us define the following conditional extremes problem   with the following notation $\overline \omega = \overline \omega_n = (\omega_1,\omega_2,\ldots,\omega_n)\in \mathds R^n$

\begin{equation*} 
\begin{aligned}
& \min_{\overline{\omega}_n\in \mathds R^n} \mathbf{E}(\overline \omega_n)\\
&\text{subject to } \omega_k - x_k = 0, \; \text{for all }\;  k=1,\ldots,m_{\mathcal E}.
\end{aligned}
\end{equation*}

Note that this problem is equivalent to solve
\begin{equation*}
\min_{\overline \omega_{n-m_\mathcal{E}} \in \mathds R^{n-m_\mathcal{E}}}  \mathbf{E}_R(x_{1},\ldots,x_{m_{\mathcal E}},\overline \omega_{n-m_\mathcal{E}}).
\end{equation*}

Let us prove that $\overline x_{n-m_\mathcal{E}}$ is a minimum of this problem. Note that the gradient of this function corresponds to the last $n-m_\mathcal{E}$ conditions of \eqref{PtosCrit-Potencial}, and the second order condition is given by the truncated Hessian $\nabla_{\omega_{m_{\mathcal{E}}}\omega_{m_{\mathcal{E}}}}^2 \mathbf{E}(\overline x_{m_{\mathcal{E}}})$, which is, by hypothesis, positive definite.

Therefore, the configuration $\overline{x}_n$ corresponds to the local equilibrium of the energy function \eqref{LogPotential} once the first $m_{\mathcal{E}}$ charges are fixed.

\end{remark}


\end{document}